\documentclass[a4paper,10pt]{article}

\usepackage[cp1251]{inputenc}
\usepackage[english]{babel}
\usepackage{amsthm}
\usepackage{cite}
\usepackage{tikz-cd}
\usepackage{amsmath}
\usepackage{amsfonts}
\usepackage{amssymb}

\usepackage{epsfig,graphicx}

\DeclareMathOperator{\Spec}{Spec}
\DeclareMathOperator{\Aut}{Aut}
\DeclareMathOperator{\End}{End}
\DeclareMathOperator{\Char}{char}
\DeclareMathOperator{\Det}{det}
\DeclareMathOperator{\DetB}{Det_p}

\DeclareMathOperator{\Mat}{Mat}
\DeclareMathOperator{\Tr}{tr}

\DeclareMathOperator{\Deg}{deg}
\DeclareMathOperator{\ord}{ord}

\newtheorem{thm}{Theorem}[section]
\newtheorem{lem}[thm]{Lemma}
\newtheorem{slem}[thm]{Sublemma}
\newtheorem{prop}[thm]{Proposition}

\newtheorem{Def}[thm]{Definition}
\newtheorem{remark}[thm]{Remark}

\begin{document}
\fontsize{12}{12pt}\selectfont
\title{\bf On Planar Algebraic Curves and Holonomic $\mathcal{D}$-modules
 in Positive Characteristic }
\author{Alexei Kanel-Belov and Andrey Elishev}

\date{}

\maketitle
\renewcommand{\abstractname}{Abstract}
\begin{abstract}
In this paper we study a correspondence between cyclic modules over the first Weyl algebra and planar algebraic curves in positive characteristic. In particular, we show that any such curve has a preimage under a morphism of certain ind-schemes. This property might pave the way to an indirect proof of existence of a canonical isomorphism between the group of algebra automorphisms of the first Weyl algebra over the field complex numbers and the group of polynomial symplectomorphisms of $\mathbb{C}^2$.
\end{abstract}

\section{Introduction}

Let $R$ be an associative unitary ring. For $n\in\mathbb{N}$ the $n$-th Weyl algebra $A_{n,R}$ over $R$ is defined as the quotient
\begin{equation*}\tag{1.1}
A_{n,R}=R\langle x_1,\ldots,x_n,y_1,\ldots,y_n\rangle/I,
\end{equation*}
with the ideal $I$ being generated by all elements of the form
\begin{equation*}
x_ix_j-x_jx_i,\; y_iy_j-y_jy_i,\; y_ix_j-x_jy_i-\delta_{ij}
\end{equation*}
for $1\leq i,j\leq n$.\\

\smallskip

The $R$-algebra $A_{n,R}$ is associative, unital $R$-algebra with $2n$ generators given by the images of $x_i,\;y_i$ under the standard projection (we will from here on denote these generators simply as $x_i,\;y_i$). It is also a free $R$-module of infinite countable rank.\\

\smallskip

The case when $R\equiv\mathbf{k}$ is a field is of main interest. The Weyl algebra then coincides with the algebra $D(\mathbb{A}^n_{\mathbf{k}})$ of polynomial differential operators on $\mathbb{A}^n_{\mathbf{k}}=\Spec\mathbf{k}[x_1,\ldots,x_n]$, with $x_i$ acting as multiplication by an indeterminate $x_i$ and $y_i$ as taking partial derivative $\frac{\partial}{\partial x_i}$.\\

\medskip

We now recall several open conjectures in which the Weyl algebra emerges, along with a number of objects naturally associated to the free algebra, the polynomial algebra and itself. The statements outlined below have been demonstrated to possess a profound interconnection as well as relation to other problems of mathematical physics.\\

\smallskip

Let $\Aut(A_{n,\mathbf{k}})$ be the group of algebra automorphisms of $A_{n,\mathbf{k}}$, and let $\End(A_{n,\mathbf{k}})$ be the monoid of algebra endomorphisms of $A_{n,\mathbf{k}}$. \\

\medskip

The \textbf{Dixmier Conjecture} $\text{DC}_n$, first stated in \cite{1}, asserts that every algebra endomorphism of $A_{n,\mathbf{k}}$ is invertible, that is, $\Aut(A_{n,\mathbf{k}})=\End(A_{n,\mathbf{k}})$, whenever $\Char\mathbf{k}=0$. By the Lefschetz principle it is sufficient to set the base field to be the field of complex numbers $\mathbb{C}$. The $\text{DC}_n$ implies $\text{DC}_m$ for all $n>m$; the conjunction $\bigwedge_{n\in\mathbb{N}}\text{DC}_n=\text{DC}_{\infty}$ is referred to as the stable Dixmier conjecture. The conjecture $\text{DC}_n$ is open for all $n\in\mathbb{N}$.\\

\medskip

\begin{remark}
A similar statement in the case of finite characteristic does not hold, at least for $n>1$. Indeed, a few years ago Bavula \cite{15} asked whether any $\mathbf{k}$-endomorphism of $A_{n,\mathbf{k}}$ ($\Char\mathbf{k}=p>0$) is injective. The negative answer to this question for $n>1$ was provided by Makar-Limanov, \cite{17}.
\end{remark}

\medskip

The \textbf{Jacobian Conjecture} $\text{JC}_n$ states that for any field $\mathbf{k}$ of characteristic zero any polynomial endomorphism $\phi$ of $\mathbb{A}^n_{\mathbf{k}}$ with unital jacobian
\begin{equation*}\tag{1.2}
\Det\left\vert\left\vert\left(\frac{\partial \phi^*(x_i)}{\partial x_j}\right)_{1\leq i,j\leq n}\right\vert\right\vert=1
\end{equation*}
is an automorphism. Again, by the Lefschetz principle one may set $\mathbf{k}=\mathbb{C}$. The $\text{JC}_n$ implies $\text{JC}_m$ whenever $n>m$, and $\text{JC}_{\infty}$ denotes the stable Jacobian conjecture. Evidently $\text{JC}_1$ is true, as linear maps are globally invertible; $\text{JC}_n$, however, is open for all $n\geq 2$. A detailed description of the Jacobian conjecture and its equivalent formulations can be found in \cite{2}.

\medskip

It is known that $\text{DC}_n\Rightarrow\text{JC}_n$, and that $\text{JC}_{2n}\Rightarrow\text{DC}_n$, which together imply that $\text{JC}$ and $\text{DC}$ are stably equivalent. The implication $\text{JC}_{2n}\Rightarrow\text{DC}_n$ is, much unlike its converse, rather non-trivial. It was shown to be true by the first author together with Maxim Kontsevich, see \cite{3}, and independently by Y. Tsuchimoto, cf. \cite{5}, \cite{13}. Both proofs in one way or another rely on reduction of the Weyl algebra to positive characteristic and on upper-bounding the degrees of the endomorphisms involved. A version of proof differing from that of \cite{3} was developed by Bavula (\cite{14},\; \cite{15}, also \cite{16}); it employs the inversion formulae for automorphisms of the Weyl algebra.

\medskip

Another, perhaps even more surprising conjecture was formulated by Kontsevich together with the first author around the same time, cf. \cite{4}, and is generally referred to as the \textbf{Belov-Kanel - Kontsevich Conjecture} $\text{B-KKC}_n$. Let
\begin{equation*}\tag{1.3}
 P_{n,\mathbb{C}}=\mathbb{C}[z_1,\ldots,z_{2n}]
\end{equation*}
be the polynomial $\mathbb{C}$-algebra over $2n$ variables equipped with the Poisson bracket
\begin{equation*}\tag{1.4}
\lbrace z_i,z_j\rbrace = \delta_{i,n+j}-\delta_{i+n,j}.
\end{equation*}
Denote by $\Aut(P_{n,\mathbb{C}})$ the group of Poisson structure-preserving automorphisms of $P_{n,\mathbb{C}}$, i.e. the group of polynomial symplectomorphisms of $\mathbb{A}^{2n}_{\mathbb{C}}$. The $\text{B-KKC}_n$ \textbf{states} that the groups $\Aut(A_{n,\mathbb{C}})$ and $\Aut(P_{n,\mathbb{C}})$ are canonically isomorphic. In effect, this statement was conjectured to hold also for automorphisms over the field of rational numbers.

\medskip

The $\text{B-KKC}_n$ is true for $n=1$. The proof is essentially a straightforward description of the groups involved: the structure of $\Aut(P_{n,\mathbb{C}})$ was obtained by Jung and van der Kulk in mid-twentieth century, (see \cite{6} and \cite{7} respectively), and is represented as the quotient of a free product of two groups as follows:
\begin{equation*}\tag{1.5}
\Aut(P_{n,\mathbb{C}})\simeq G_1 * G_2/(G_1\cap G_2),
\end{equation*}
where
$G_1\simeq SL(2,\mathbb{C})\ltimes\mathbb{C}^2$
is the special affine group and $G_2$ is a group of polynomial transformations of the form
\begin{equation*}
(x_1,x_2)\mapsto(ax_1+f(x_2),a^{-1}x_2),\;\;a\in\mathbb{C}^{\times},\;\;f\in\mathbb{C}[z].
\end{equation*}
Thirty years ago Makar-Limanov, \cite{8}, \cite{9}, showed that the automorphism groups of the corresponding Weyl algebra and the free associative algebra in two variables admit a similar description. Subsequently, the case $n=1$ is resolved positively. Higher-dimensional case is open to this day; however, for a substantial class of automorphisms the canonical isomorphism does exist (the groups of so-called tame automorphisms are canonically isomorphic for all $n$ - see \cite{4}, specifically Section 7). \\

\bigskip

\begin{remark}
In the case $n=1$ and in characteristic zero, there are other canonical isomorphisms in an analogous setting - particularly, between $\Aut \mathbf{k}[x,y]$, $\Aut \mathbf{k}\langle x,y\rangle$, and $\Aut \mathbf{k}\lbrace x,y\rbrace$, the automorphism groups of the polynomial, free associative and free Poisson algebra in two variables, respectively - see \cite{18}. The known way of establishing that relies heavily on the fact that all those automorphisms are tame. This convenient property ceases to exist in higher dimensions, cf. \cite{19}.
\end{remark}

\bigskip

It is reasonable to expect the conjectured canonical isomorphism to coincide with the aforementioned isomorphism on the set of tame automorphisms, although there is also reason to anticipate a rather complicated object. It seems pertinent to look for an indirect proof, something along the lines of Section 8 in \cite{4}. The present paper aims to explore a particular means of doing so, in the known case $n=1$; however, the regarded indirect approach is still to be proven in general, as - at the very least technical - difficulties arise when one increases the dimension of the underlying phase space (not to mention the fact that a generic holonomic $\mathcal{D}$-module is a fairly straightforward object only in the case $n=1$).

\section{Correspondence between holonomic $\mathcal{D}$-modules \\
and lagrangian subvarieties}
In this section we recall the notion of a lagrangian submanifold of a symplectic manifold as well as describe how symplectomorphisms are naturally associated with lagrangian submanifolds.

\medskip

Let $P$ be a smooth real or complex manifold, $TP$ be its tangent bundle. A symplectic structure $\Omega$ on a smooth manifold $P$ is any closed nondegenerate differential $2$-form on $P$.

\medskip

The simplest example is given by the $2n$-dimensional complex affine space $\mathbb{A}^{2n}_{\mathbb{C}}$ with the symplectic structure $\Omega$ expressed in local coordinates $(x_1,\ldots,x_n,p_1,\ldots,p_n)$ as
\begin{equation*}
\Omega=\sum_{i=1}^{n}dx_i\wedge dp_i.
\end{equation*}

\medskip

Another archetypal example would be that of a finite-dimensional real vector space (or, slightly more generally, a free module of finite rank over an arbitrary commutative ring). If $V$ is a vector space, then a symplectic structure $\Omega$ is a skew-symmetric bilinear form on $V$, such that the associated map to the dual space $V^*$
\begin{equation*}
\tilde{\Omega}:V\rightarrow V^*,\;\;v\mapsto \tilde{\Omega}(v):\;\tilde{\Omega}(v)(w)=\Omega(v,w)
\end{equation*}
is an isomorphism. Making use of this point of view, one may define the symplectic structure on a smooth manifold $P$ as a smooth family $\Omega=\lbrace \Omega_p\rbrace$ of symplectic structures on the fibers of the tangent bundle $TP$.

\bigskip

Certain types of submanifolds of symplectic manifolds arise naturally from the notion of orthogonality.

\medskip

Let $(V,\;\Omega)$ be a symplectic vector space in the sense of the above definition. For any subspace $W\subseteq V$, the orthogonal complement $W^{\perp}$ is defined as the subspace
\begin{equation*}
\lbrace v\in V\;|\;\Omega(v,w)=0,\;\forall w\in W\rbrace.
\end{equation*}

\medskip

It is easy to see that $\dim(W)+\dim(W^{\perp})=\dim(V)$.
\medskip

\medskip

One then has the following key definitions:
\begin{Def}
$W$ is called isotropic if $W\subseteq W^{\perp}$. $W$ is called coisotropic if $W\supseteq W^{\perp}$.
\end{Def}

Note that $W$ is isotropic iff the restriction $\Omega_{|W}$ is identically zero.

\begin{Def}
A subspace $W\subseteq V$ is lagrangian if it is both isotropic and coisotropic.
\end{Def}

It follows that for a lagrangian subspace $W$, $\dim(W)=\frac1{2}\dim{V}$.
\medskip

Let $Q$ be a submanifold of $P$. Its tangent bundle $TQ$ may be viewed as a subbundle of $T_QP=\cup_{q\in Q}T_qP$. The notions of isotropic, coisotropic and lagrangian subspace are extended to the case of submanifolds in an obvious way. Namely, $Q$ is called an isotropic [coisotropic] submanifold of $P$ if for any $q\in Q$ the space $T_qQ$ is an isotropic [coisotropic] subspace of $T_qP$. $Q$ is a lagrangian submanifold of $P$ if it is both isotropic and coisotropic.

\bigskip

Apart from possessing quite a number of remarkable properties, lagrangian submanifolds naturally arise from symplectomorpishs.
Consider the case $P=\mathbb{A}_{\mathbb{C}}^{2n}$ with the symplectic structure as in the example above.
Let $\psi\in\Aut(P_{n,\mathbb{C}})$ be a smooth symplectomorphism, that is
\begin{equation*}\tag{2.1}
\psi:\mathbb{A}^{2n}_{\mathbb{C}}\rightarrow \mathbb{A}^{2n}_{\mathbb{C}}
\end{equation*}
is an isomorphism preserving the symplectic structure $\Omega=\sum_{i=1}^ndx_i\wedge dp_i$. \\
\medskip

For an affine symplectic manifold $P=(\mathbb{A}^{2n}_{\mathbb{C}},\;\Omega)$, the manifold $\bar{P}=(\mathbb{A}^{2n}_{\mathbb{C}},\;\bar{\Omega})$, $\bar{\Omega}=-\Omega$ is called the dual manifold of $P$. Consider the (tensor) product $P\times\bar{P}=(\mathbb{A}^{4n}_{\mathbb{C}},\;\pi_1^*\Omega+\pi_2^*\bar{\Omega})$ with $\pi_i^*$ being the duals of cartesian projections
\begin{equation*}\tag{2.2}
\pi_i:\mathbb{A}^{2n}_{\mathbb{C}}\times \mathbb{A}^{2n}_{\mathbb{C}}\rightarrow \mathbb{A}^{2n}_{\mathbb{C}},\;\;i=1,2.
\end{equation*}\\
Then the graph of $\psi$, i.e. the set
\begin{equation*}
\Gamma_{\psi}=\lbrace (z,\;\psi(z))\;|\;z\in \mathbb{A}^{2n}_{\mathbb{C}}\rbrace
\end{equation*}
is a lagrangian submanifold of $P\times\bar{P}$.
\medskip

More generally, one has the following \\
\begin{prop}
 Let $P_1=(X_1,\;\Omega_1),\;P_2=(X_2,\;\Omega_2)$ be smooth manifolds equipped with symplectic structures $\Omega_1, \;\Omega_2$ respectively. For every smooth symplectomorphism $\psi:P_1\rightarrow P_2$ its graph
\begin{equation*}\tag{2.3}
\Gamma_{\psi}=\lbrace (z,\;\psi(z))\;|\;z\in X_1\rbrace
\end{equation*}
is a lagrangian submanifold of $P_1\times \bar{P_2}$ (where $\bar{P_2}$ is the dual manifold of $P_2$).
\end{prop}
\medskip
\begin{proof}
Indeed, $\Gamma_{\psi}$ is isotropic:
\begin{gather*}
\Gamma_{\psi}^{\perp}\equiv\lbrace w\in\Gamma_{\psi}\;|\;\Omega(w,w')=0,\;\forall w'\in\Gamma_\psi\rbrace,\\
\forall w=(z,\;\psi(z))\in\Gamma_\psi,\;w'=(z',\;\psi(z'))\in\Gamma_\psi\Rightarrow \\
\Omega(w,w')=\Omega_1(z,z')-\Omega_2(\psi(z),\psi(z'))=0,
\end{gather*}
because $\psi$ preserves the symplectic structure. \\

Then, $\Gamma_\psi$ is also coisotropic. Take any $v=(z_1,\;z_2)$ in the orthogonal complement $\Gamma_{\psi}^{\perp}$ and any $w=(z,\;\psi(z))$ in $\Gamma_\psi$. By definition
\begin{equation*}
\Omega(v,w)=0,
\end{equation*}
so that
\begin{equation*}
\Omega_1(z_1,z)=\Omega_2(z_2,\psi(z)).
\end{equation*}
As $\psi$ is an isomorphism, $\exists!\; y\in P_1:\;z_2=\psi(y)$, but then
\begin{equation*}
\Omega_1(y-z_1,z)=0,
\end{equation*}
and $y=z_1$ follows from non-degeneracy.
\end{proof}

\medskip
\begin{remark}
Note that the converse is also true. \\
\end{remark}
\medskip

Now, by this proposition, any polynomial symplectomorphism $\psi\in\Aut(P_{n,\mathbb{C}})$ corresponds to a lagrangian subvariety $L_\psi$ of the form (2.3) in $\mathbb{A}^{4n}_{\mathbb{C}}$ endowed with symplectic structure formed from that of $\mathbb{A}^{2n}_{\mathbb{C}}$ as above.\\

\medskip

One can establish along the lines of Proposition 2.3 a similar statement involving polynomial symplectomorphisms of affine space in positive characteristic. \\

\medskip

As noted in \cite{4}, any automorphism $\phi\in\Aut(A_{n,\mathbf{k}})$ gives a bimodule, which can be viewed as a holonomic $A_{2n,\mathbf{k}}$-module $M_\phi$. It might be possible to arrive at the $\text{B-KKC}_n$ by establishing a canonical correspondence between such modules and lagrangian subvarieties of $\mathbb{A}^{4n}_{\mathbf{k}}$, thus constructing an inverse map by means of appropriate lifting to characteristic zero
\begin{equation*}\tag{2.4}
\Aut(P_{n,\mathbb{C}})\rightarrow \Aut(A_{n,\mathbb{C}}).
\end{equation*}

\medskip

The base case $n=1$ seems to be penetrable in positive characteristic. Any lagrangian subvariety of $\mathbb{A}^{4}_{\mathbf{k}}$ has dimension 2 and therefore corresponds to a system of two polynomial equations of the form
\begin{gather*}\tag{2.5}
f_1(x_1,x_2,y_1,y_2)=0,\\
f_2(x_1,x_2,y_1,y_2)=0.
\end{gather*}

Then, the subvariety can be projected separately onto subspaces spanned by \\
$(x_1,x_2,y_1)$ and $(x_1,x_2,y_2)$, leading in each case to a set of planar curves parameterised by $x_2$. In the next section we establish that any such curve corresponds to at least one differential operator, so that the projections of (2.5) would naturally generate elements of $A_{2,\mathbf{k}}$. However, it still remains to be seen whether the appropriate $A_{2,\mathbf{k}}$-module is non-trivial. This problem is closely related to a situation which emerges when one considers the integrability of certain systems of partial differential equations.

\section{Weyl algebra and planar curves \\
in positive characteristic}
Let $\Char \mathbf{k}=p>0$ and let $A_{1,\mathbf{k}}$ be the first Weyl algebra over $\mathbf{k}$. It as a free module of rank $p^2$ over its center $C(A_{1,\mathbf{k}})=\mathbf{k}\left[x^p,y^p\right]$ together with the standard basis $\mathfrak{B}_1=\lbrace x^iy^j\;|\;0\leq i,j\leq p-1 \rbrace$. A key observation is that $A_{1,\mathbf{k}}$ is an Azumaya algebra of rank $p$ over $\mathbf{k}[x_1,x_2]$ (and, in general, that $A_{n,\mathbf{k}}$ is Azumaya of rank $p^n$ over the polynomial algebra in $2n$ variables). In particular, we perform the following procedure. \

\medskip

Let us extend the center $\mathbf{k}\left[x^p,y^p\right]$ of the first Weyl algebra by adding central variables $\tilde{x},\;\tilde{y}$, such that

\begin{equation*}\tag{3.1}
\tilde{x}^p=x^p,\;\tilde{y}^p=y^p.
\end{equation*}
One has now the following lemma:

\medskip

\begin{lem}
The extension of $A_{1,\mathbf{k}}$ is isomorphic to the matrix algebra
\begin{equation*}\tag{3.2}
\Mat \left( p\times p,\; \mathbf{k}\left[\tilde{x},\tilde{y}\right]\right).
\end{equation*}
\end{lem}

\medskip

\begin{proof}
 Indeed, for any prime $p$ the unital algebra $A$ over $\mathbf{k}$, ($\Char\mathbf{k}=p$) generated by two elements $y_1,\;y_2$ satisfying the relations
\begin{equation*}\tag{3.3}
[y_1,y_2]=1,\;y_1^p=y_2^p=0
\end{equation*}
is isomorphic to $\Mat \left( p\times p,\; \mathbf{k}\right)$. Namely, a direct calculation shows there is an isomorphism
\begin{equation*}
A\rightarrow \End_{\mathbf{k}-\text{mod}}(\mathbf{k}[x]/(x^p))
\end{equation*}
given by
\begin{equation*}
y_1\mapsto d/dx,\;\;y_2\mapsto x
\end{equation*}
(differentiation and multiplication by $x$ operators).
\smallskip

Now, the extension
\begin{equation*}
\mathbf{k}[\tilde{x}^p,\tilde{y}^p]\rightarrow \mathbf{k}[\tilde{x},\tilde{y}]
\end{equation*}
is faithfully flat, as $\mathbf{k}[\tilde{x},\tilde{y}]$ is a flat $\mathbf{k}[\tilde{x}^p,\tilde{y}^p]$-algebra and the induced morphism
\begin{equation*}
\Spec \mathbf{k}[\tilde{x},\tilde{y}]\rightarrow \Spec \mathbf{k}[\tilde{x}^p,\tilde{y}^p]
\end{equation*}
is surjective.

\smallskip

The pullback algebra
\begin{equation*}\tag{3.4}
A=A_{1,\mathbf{k}}\otimes_{\mathbf{k}[\tilde{x}^p,\tilde{y}^p]}\mathbf{k}[\tilde{x},\tilde{y}]
\end{equation*}
as a $\mathbf{k}[\tilde{x},\tilde{y}]$-algebra is generated by $x,\;y$ with relations
\begin{equation*}\tag{3.5}
[y,x]=1,\;x^p=\tilde{x}^p,\;y^p=\tilde{y}^p.
\end{equation*}
Shifted generators $x'=x-\tilde{x},\;y'=y-\tilde{y}$ satisfy
\begin{equation*}\tag{3.6}
[y',x']=1,\;(x')^p=(y')^p=0,
\end{equation*}
which means that $A$ is isomorphic to the $\mathbf{k}$-tensor product of $\mathbf{k}[\tilde{x},\tilde{y}]$ with a $\mathbf{k}$-algebra generated by $x',\;y'$ as above. The remark at the beginning of the proof yields the statement of the lemma.
\end{proof}

\medskip

The above statement is a special case of a more general construction provided in \cite{3}. Also, a similar result holds for rings of differential operators over smooth affine schemes $X$ in positive characteristic, with the Weyl algebra emerging in the case $X=\mathbb{A}^n_{\mathbf{k}}$ (cf. \cite{11}).\\

\bigskip

The structure of a generic holonomic left $A_{1,\mathbf{k}}$-module is well known - any such module is of the form
\begin{equation*}
M=A_{1,\mathbf{k}}/A_{1,\mathbf{k}}\cdot L
\end{equation*}
for some differential operator $L=\sum_{i+j\leq N}a_{ij}x^iy^j\in A_{1,\mathbf{k}}$.   \\

\medskip

Let $X_p,\;Y_p$ denote the $p\times p$ matrices representing the shifted generators (3.6), $I_p$ be the identity matrix of size $p\times p$. Following \cite{11}, consider for every differential operator $L=\sum_{i+j\leq N}a_{ij}x^iy^j$ the $p$-determinant polynomial
\begin{equation*}\tag{3.7}
\DetB L:=\Det\left(\sum_{i+j\leq N}a_{ij}(X_p+\tilde{x}I_p)^i(Y_p+\tilde{y}I_p)^j\right).
\end{equation*}

\medskip

\begin{lem}
 For $L=\sum_{i+j\leq N}a_{ij}x^iy^j$ the corresponding $p$-determinant $\DetB L$ is a polynomial in $\tilde{x}^p,\;\tilde{y}^p$ of degree $\leq N$.\footnote{This statement was communicated to us by Maxim Kontsevich, also cf. \cite{11}}

\end{lem}
\medskip
\begin{proof}
 It suffices to show that the $p$-determinant vanishes after taking partial derivatives
\begin{equation*}\tag{3.8}
\frac{\partial}{\partial\tilde{x}}\DetB L=\frac{\partial}{\partial\tilde{y}}\DetB L=0.
\end{equation*}

The proof of that particular statement is elementary and reduces to an application of a well-known result of Jacobi, namely that
\begin{equation*}\tag{3.9}
\frac{\partial \det(A)}{\partial x}=\Tr(\text{adj}(A)\frac{\partial A}{\partial x}),
\end{equation*}
where "$\text{tr}$" \; means taking the trace of a matrix, and $\text{adj}(A)$ is the adjugate matrix of $A$. Note that the adjugate can be expressed as a finite sum of powers of $A$: if $P_A(\lambda)=\det(A-\lambda I)$ is the characteristic polynomial of $A$ and $f_A(t)=(P_A(0)-P_A(\lambda))/\lambda$, then
\begin{equation*}
\text{adj}(A)=f_A(A).
\end{equation*}
In particular, the adjugate commutes with $A$.

\medskip

Let us consider the derivative with respect to $\tilde{x}$. Just as in characteristic zero, we have
\begin{equation*}\tag{3.10}
[Y_p,X_p^k]=kX_p^{k-1},
\end{equation*}
so that taking the derivative of the matrix $A(\tilde{x},\tilde{y})$ from the $p$-determinant expression is equivalent to taking the commutator with $Y_p$:
\begin{equation*}\tag{3.11}
\frac{\partial A}{\partial \tilde{x}}=-[A,Y_p].
\end{equation*}
Plugging it in (3.9) yields
\begin{equation*}\tag{3.12}
\frac{\partial \det(A)}{\partial \tilde{x}}=-\Tr(\text{adj}(A)[A,Y_p]),
\end{equation*}
which equals zero by the cyclic property of the trace (applied after one permutes the commuting $\text{adj}(A)$ and $A$). The vanishing of the derivative with respect to $\tilde{y}$ is shown similarly.
\end{proof}

\bigskip

We are also going to need the following observation:
\medskip

\begin{lem}
 If $L\neq 0$ then $\DetB L\neq 0$.
\end{lem}

\begin{proof}
 The statement is true if $L$ has degree zero.

Suppose $\Deg L =N >0$. Let $(i_0,\;j_0)$ correspond to the leading term in the $(x,\;y)$-lexicographical ordering of all the monomials in $L$ of highest total degree (that is, $i_0+j_0=N$ and all the other degree $N$ terms in $L$ have $\deg_x$ less than $i_0$). The $p$-determinant of this term is easy to evaluate - it is the product of powers of characteristic polynomials of nilpotent matrices $X_p,\;Y_p$ (over the variables $\tilde{x}$ and $\tilde{y}$, respectively), multiplied by $a_{i_0j_0}^p$, and hence is equal to $a_{i_0j_0}^p\tilde{x}^{pi_0}\tilde{y}^{pj_0}$.

\smallskip

It suffices to show that the $p$-determinant of $L$, viewed as a polynomial in $\tilde{x}$ and $\tilde{y}$, contains a term $\tilde{x}^{pi_0}\tilde{y}^{pj_0}$ with the coefficient being precisely $a_{i_0j_0}^p$. In order to proceed, we are going to need the following technical sublemma.
\smallskip

\begin{slem}
For an arbitrary collection $\lbrace A^1,\ldots,A^m\rbrace$ of $p\times p$-matrices the determinant of its sum decomposes into a sum of determinants as follows:
\begin{equation*}\tag{3.13}
\Det (A^1+\cdots+A^m) =\sum_{\sigma}\Det(A^\sigma),
\end{equation*}
where the sum extends over all maps $\sigma:\lbrace 1,\ldots,p\rbrace\rightarrow \lbrace 1,\ldots,m\rbrace$, and the entries of $A^\sigma$ are
\begin{equation*}\tag{3.14}
(A^\sigma)_{ij}=A^{\sigma(i)}_{ij}.
\end{equation*}
\end{slem}
\begin{proof}
Indeed, let $A=A^1+\cdots+A^m$. Expanding the determinant of $A$ along the first row yields
\begin{equation*}
\det A = \sum_{i=1}^{m}\det A^{(i)},
\end{equation*}
where the first row of $A^{(i)}$ is assembled out of elements of $A^i$, and the remaining rows are from $A$. Each $\det A^{(i)}$ is then expanded along the second row:
\begin{equation*}
\det A^{(i)} = \sum_{j=1}^{m}\det A^{(i,j)},
\end{equation*}
where in $A^{(i,j)}$ the first row is the first row of $A^i$, the second one is the second row of $A^j$, so that
\begin{equation*}
\det A = \sum_{(i,j)\in [m]\times [m]} \det A^{(i,j)}.
\end{equation*}
Iterating the process, we obtain
\begin{equation*}
\det A = \sum_{1\leq i_1,\ldots,i_p\leq m} \det A^{(i_1,\ldots,i_p)}.
\end{equation*}
Each term $(i_1,\ldots,i_p)\in [m]^{\times p}$ defines a map $\sigma:\lbrace 1,\ldots,p\rbrace\rightarrow \lbrace 1,\ldots,m\rbrace$, likewise every map $[p]\rightarrow [m]$ is present. The k-th row of the matrix $A^{(i_1,\ldots,i_p)}$ is the k-th row of $A^{i_k}$, which implies $A^{(i_1,\ldots,i_p)}=A^{\sigma}$. Sublemma is proved.\\
\end{proof}
\smallskip


The $p$-determinant is the determinant of the matrix
\begin{equation*}
\sum_{i+j\leq N}a_{ij}(X_p+\tilde{x}I_p)^i(Y_p+\tilde{y}I_p)^j,
\end{equation*}
which in turn is a polynomial in $\tilde{x}$ and $\tilde{y}$ with matrix coefficients. Let us rewrite this matrix as a sum $A^1+\cdots+A^m$, where all $A^k,\;k=1,\ldots,m$ are monomials in $\tilde{x}$ and $\tilde{y}$, and then apply the sublemma.

By definition of $(i_0,\;j_0)$, the term
\begin{equation*}
A^m=a_{i_0j_0}\tilde{x}^{i_0}\tilde{y}^{j_0}I_p,
\end{equation*}
which appears after one expands all of the binomials $(X_p+\tilde{x}I_p)^i$ and $(Y_p+\tilde{y}I_p)^j$, is of highest total degree $N$ and also $(\tilde{x},\tilde{y})$-lexicographically greater than all the other terms of total degree $N$. The determinant of this term is equal to $a_{i_0j_0}^p\tilde{x}^{pi_0}\tilde{y}^{pj_0}$ as noted above - we just need to prove that nothing can annihilate it in the decomposition obtained from the sublemma. Indeed, summands $A^k$ with total degree less than $N$ cannot generate a determinant term $\Det(A^{\sigma})$ of total degree $pN$, neither can their rows be mixed with the rows of degree $N$ terms to produce a $\Det(A^{\sigma})$ of degree $pN$. In fact, the only way to match the $pN$ is to combine rows taken exclusively from degree $N$ summands. All such summands are of the form $a_{ij}\tilde{x}^{i}\tilde{y}^{j}I_p$ with $i+j=N$ (diagonal matrices), but then in order to reach the highest lexicographical term one must take all $p$ rows from $a_{i_0j_0}\tilde{x}^{i_0}\tilde{y}^{j_0}I_p$, in other words, set $A^{\sigma}=A^m$. Therefore, nothing can neutralize the $a_{i_0j_0}^p\tilde{x}^{pi_0}\tilde{y}^{pj_0}$ term in the $p$-determinant of $L$, and the lemma is proved.
\end{proof}

\bigskip

Let $\mathbf{k}$ be an algebraically closed field. Denote by $\mathfrak{C}(\mathbb{A}_{\mathbf{k}}^2)$ the set of all algebraic curves in $\mathbb{A}_{\mathbf{k}}^2=\Spec \mathbf{k}[z_1,z_2]$. This set possesses a natural structure of an ind-scheme. Indeed, any curve in $\mathfrak{C}(\mathbb{A}_{\mathbf{k}}^2)$ is the zero-locus of some polynomial in two variables. The correspondence between curves and polynomials is one-to-one modulo multiplicative constant, therefore the set of curves of degree $\leq d$ is the projective space  $\mathbb{P}^{\frac{(d+1)(d+2)}{2}-1}$ over $\mathbf{k}$. Indexed by $d\in\mathbb{N}$, these spaces together with obvious embeddings form an inductive system with a direct limit
\begin{equation*}\tag{3.15}
\mathfrak{C}(\mathbb{A}_{\mathbf{k}}^2)=\varinjlim\mathbb{P}^{\frac{(d+1)(d+2)}{2}-1}.
\end{equation*}

\medskip

This expression can be viewed as the definition of the infinite-dimensional projective space.\\

\medskip

The first Weyl algebra $A_{1,\mathbf{k}}$ together with the Bernstein filtration
\begin{gather*}\tag{3.16}
\mathfrak{F}=\lbrace F_d\;|\;d\in\mathbb{Z}_+\rbrace, \\
F_d=\lbrace L\in A_{1,\mathbf{k}}\;|\;\Deg_x L+\Deg_y L\leq d\rbrace,\\
\Deg_x(x)=\Deg_y(y)=1;
\end{gather*}
can be turned into an infinite-dimensional projective space by the equivalence relation $\mathfrak{F}/\mathbf{k}^\times$, which glues together elements differing by a non-zero factor. Note that its dimensional structure as a direct limit is the same as that of the ind-scheme $\mathfrak{C}(\mathbb{A}_{\mathbf{k}}^2)$.\\

\medskip

Let $\mathcal{F}_d/\mathbf{k}^\times$ be the projective space of classes of differential operators of degree at most $d$, and let $\mathfrak{C}_d$ be the space of algebraic curves in $\mathbb{A}^2_{\mathbf{k}}$ of degree at most $d$.

\medskip

By Lemma 3.2, for every $d\in\mathbb{N}$ the $p$-determinant map induces a morphism
\begin{equation*}\tag{3.17}
\Theta_d:\mathcal{F}_d/\mathbf{k}^\times\rightarrow\mathfrak{C}_d ,\;\;[L]\mapsto\mathit{C}_L\subset \mathbb{A}_{\mathbf{k}}^2,
\end{equation*}
such that $C_L$ is the zero locus of the polynomial $\DetB(L)$ (up to a multiplicative constant) in variables $z_1=\tilde{x}^p$ and $z_2=\tilde{y}^p$.
\medskip

For any pair $(d,\;d')$ such that $d\leq d'$ denote by
\begin{equation*}
\mu_{dd'}:\mathcal{F}_d/\mathbf{k}^\times\rightarrow\mathcal{F}_{d'}/\mathbf{k}^\times
\end{equation*}
and
\begin{equation*}
\nu_{dd'}:\mathfrak{C}_d\rightarrow\mathfrak{C}_{d'}
\end{equation*}
the obvious embeddings that together with the sets $\mathcal{F}_d/\mathbf{k}^\times$ and $\mathfrak{C}_d$ respectively form the inductive systems described above. Clearly these embeddings are closed in the Zariski topology.

\medskip

The following proposition is elementary:

\begin{prop}
For all $d,d'\in\mathbb{N}$ the diagram
\begin{equation*}\tag{3.18}
\begin{tikzcd}
 \mathcal{F}_d/\mathbf{k}^\times \arrow{r}{\Theta_d} \arrow{d}{\mu_{dd'}} & \mathfrak{C}_d\arrow{d}{\nu_{dd'}}\\
\mathcal{F}_{d'}/\mathbf{k}^\times\arrow{r}{\Theta_{d'}} & \mathfrak{C}_{d'}
\end{tikzcd}
\end{equation*}
is commutative.
\end{prop}

\medskip

As a corollary,  there is a well-defined morphism of ind-schemes
\begin{equation*}\tag{3.19}
\Theta:A_{1,\mathbf{k}}/\mathbf{k}^\times\rightarrow \mathfrak{C}(\mathbb{A}_{\mathbf{k}}^2),
\end{equation*}
that maps projective classes of differential operators to planar algebraic curves.

\bigskip

Our main result reduces to the following statement:

\medskip

\begin{thm}
 $\Theta$ is surjective.
\end{thm}

\begin{proof}
 It is a well-known fact (see, for instance, \cite{12}) that any dominant morphism of finite-dimensional projective varieties with discrete fibers is surjective. By definition, $\Theta$ is in effect a system $\lbrace \Theta_d\rbrace_{d\in\mathbb{N}}$ of morphisms of finite-dimensional projective varieties. In order to prove the statement of our theorem it is therefore sufficient to show that all $\Theta_d$ are dominant morphisms with discrete fibers.

\medskip

For any $d\in\mathbb{N}$ the morphism $\Theta_d$ is formed by compactifying a morphism
\begin{equation*}\tag{3.20}
\Theta_d^{\text{aff}}:\mathbb{A}^{\frac{(d+1)(d+2)}{2}}_\mathbf{k}\rightarrow\mathbb{A}^{\frac{(d+1)(d+2)}{2}}_\mathbf{k},
\end{equation*}
which is essentially the same $p$-determinant. The induced homomorphism of coordinate rings is injective by Lemma 3.3, which is equivalent to the image of $\Theta_d^{\text{aff}}$ being dense in the Zariski topology.

\medskip

It now remains to show that a generic fiber $\Theta_d^{-1}(\mathit{C})$ of a planar curve $\mathit{C}$ under the map $\Theta_d$ is discrete. That will be the case if it is zero-dimensional.

\medskip

We proceed as follows. Assuming the contrary, we choose an arbitrary projective curve inside the fiber. In order to arrive at a contradiction, we need to prove that $\Theta_d$ maps this curve to a set (of planar algebraic curves) that consists of more than one point. We first demonstrate that the $p$-determinant maps the curve to a proper parametric family of polynomials in $\tilde{x}^p$ and $\tilde{y}^p$. It then suffices to show that this family has elements that are not projectively equivalent. Both steps rely heavily on Lemma 3.3 and Sublemma 3.4.

\medskip

Let $\dim\Theta_d^{-1}(\mathit{C})>0$. Then there exists an algebraic subvariety $\mathfrak{L}\subseteq\Theta_d^{-1}(\mathit{C})$ of dimension one. For any class of differential operator $L=\sum_{i+j\leq d}a_{ij}x^iy^j$ the coefficients $a_{ij}$ assume the role of homogeneous coordinates in $\mathcal{F}_d/\mathbf{k}^\times$. The curve $\mathfrak{L}$ can then be represented as a one-parametric family of differential operators given by
\begin{equation*}\tag{3.21}
L(t)=\sum_{i+j\leq d}a_{ij}(t)x^iy^j,
\end{equation*}
with all $a_{ij}(t)$ being meromorphic functions. Here $t$ is an element transcendental over $\mathbf{k}$, and $\mathbf{k}(t)\supset\mathbf{k}$ is a transcendental extension of degree one. The coordinate functions $a_{ij}$ are elements of a finite extension $K$ of $\mathbf{k}(t)$, which is the function field of the curve $\mathfrak{L}$.

\medskip

Any meromorphic function $a\in K$ that is not constant (i.e. $a\notin\mathbf{k}$) has a pole corresponding to some point $L_0\in\mathfrak{L}$. By our assumption, $\mathfrak{L}$ has dimension one, therefore not all $a_{ij}(t)$ are constant functions. Therefore, there is a point $L_0\in\mathfrak{L}$ such that some of the functions $a_{ij}$ have a pole at that point. Let $\tau$ be the local parameter at $L_0$ - that is, $\tau$ is the generator of the (unique) maximal ideal $\mathfrak{m}_{L_0}$ of the discrete valuation ring $\mathcal{O}_{\mathfrak{L},L_0}\subset K$ corresponding to the point $L_0$ (compare this situation with the purely algebraic setting as described, for example, in \cite{20}). Functions $a_{ij}$ admit the Laurent decomposition
\begin{equation*}\tag{3.22}
a_{ij}=\frac{a_{-k(ij)}(ij)}{\tau^{k(ij)}}+\cdots+\frac{a_{-1}(ij)}{\tau}+a_0(ij)+a_1(ij)\tau+\cdots.
\end{equation*}

\medskip

Let $k$ be the largest degree of the pole at $L_0$ among all $a_{ij}$, that is
\begin{equation*}\tag{3.23}
k=\max\lbrace k(ij)\;|\;k(ij)=\ord_{L_0}a_{ij}\rbrace,
\end{equation*}
where $k(ij)$ are as in (3.22). By assumption $k\geq 1$. We may plug in the Laurent decompositions of $a_{ij}$ into (3.21) and group the terms according to the powers of $\tau$ to obtain the decomposition of $L(t)$ near $L_0$:
\begin{equation*}\tag{3.24}
L(\tau,L_0)=\frac{A_{-k}}{\tau^k}+\cdots+\frac{A_{-1}}{\tau}+A_0+A_1\tau+\cdots.
\end{equation*}
By our construction all $A_l$, $l\geq -k$ belong to the term $\mathcal{F}_d/\mathbf{k}^\times$ of the inductive system, and $A_{-k}\neq 0$.

\medskip

By the natural power-series extension of the Sublemma 3.4 from the previous section, the $p$-determinant map sends $L(\tau,L_0)$ to a family of polynomials parametrized by $\tau$:
\begin{equation*}\tag{3.25}
\DetB(L(\tau,L_0))=\frac{\DetB(A_{-k})}{\tau^{pk}}+\cdots,
\end{equation*}
where dots denote the terms with $\tau^l$, $l>-k$, and $\DetB(A_{-k})\neq 0$ by Lemma 3.3. This shows that the image of $L(t)$ under the $p$-determinant is a one-parametric family of polynomials consisting of more than one point.

\medskip

Furthermore, this family never contracts to a single projective class. To see this, consider the coefficient $a_{i_0j_0}(t)$ of the highest term $x^{i_0}y^{j_0}$ in $L(t)$ with respect to the lexicographical order and divide everything by it. Clearly, this operation does not change the curve $\mathfrak{L}$. Now, consider
\begin{equation*}\tag{3.26}
L_1(t)=x^{i_0}y^{j_0}+\sum_{\text{lower-order terms}}b_{ij}(t)x^iy^j=x^{i_0}y^{j_0}+L'(t),
\end{equation*}
where $b_{ij}(t)=\frac{a_{ij}(t)}{a_{i_0j_0}(t)}$. Applying the Sublemma 3.4 to the last expression, we see that the resulting polynomial consists of the highest-order term (with respect to the lexicographical ordering) $x^{pi_0}y^{pj_0}$ with coefficient equal to one, lower-order terms generally dependent on $t$, and the $p$-determinant of $L'(t)$. The latter is processed as before, after which we obtain a one-parametric family of polynomials which are not projectively equivalent.
\medskip

Therefore, $\Theta_d$ maps the curve $\mathfrak{L}=L(t)$ to a one-parametric family of planar algebraic curves $\mathcal{C}(t)$, and we have arrived at a contradiction.

\end{proof}

\section{Conclusion}
By establishing the fact that the map (3.17) is surjective, we are looking at a way of associating to every polynomial symplectomorphism $\psi\in\Aut(P_{1,\mathbf{k}})$ an algebra automorphism of $A_{1,\mathbf{k}}$ in positive characteristic. After that we are expecting to find a lift to characteristic zero. Combined with the results in the opposite direction which are described in \cite{10}, this could lead to a proof of the $\text{B-KKC}_1$ suitable for a generalization to a higher-dimensional case. However, issues beyond technical remain (see end of Section 2), which necessitates further investigation.

\section*{Acknowledgements}
The authors would like to thank the referee for careful reading of the paper as well as for helpful comments and remarks.

This work is supported by the RFBR grant 14-01-00548.

\bigskip

\textbf{Addresses:}

A.B.-K.: Mathematics Department, Bar-Ilan University, Ramat-Gan, 52900, Israel

kanel@mccme.ru\\

A.E.: Department of Discrete Mathematics, Moscow Institute of Physics and Technology, 9 Institutskiy per., Dolgoprudny, Moscow Region, 141700, Russian Federation

elishev@phystech.edu

\end{document}